\documentclass[12pt]{amsart}
\usepackage{latexsym,amsmath,amsfonts,amscd,amssymb}
\usepackage{graphics}
\newtheorem{theorem}{Theorem}[section]
\newtheorem{lemma}[theorem]{Lemma}
\newtheorem{proposition}[theorem]{Proposition}
\newtheorem{corollary}[theorem]{Corollary}

\newtheorem{definition}[theorem]{Definition}
.

\theoremstyle{remark}

\newcommand{\Li}{\operatorname{Li}}
\newcommand{\li}{\operatorname{li}}
\newcommand{\Res}{\operatorname{Res}}
\newcommand{\Slim}{\operatorname{S-lim}}

\newcommand{\cA}{{\mathcal A}}

\newcommand{\cO}{{\mathcal O}}

\newcommand{\cR}{{\mathcal R}}
\newcommand{\cS}{{\mathcal S}}

\newcommand{\CC}{{\mathbb C}}

\newcommand{\PP}{{\mathbb P}}
\newcommand{\QQ}{{\mathbb Q}}
\newcommand{\RR}{{\mathbb R}}

\newcommand{\ZZ}{{\mathbb Z}}

\renewcommand{\a}{\alpha}
\renewcommand{\b}{\beta}

\newcommand{\g}{\gamma}

\newcommand{\eps}{\epsilon}

\begin{document}
\title[Monodromies of singularities of the Hadamard and e\~ne product]
{Monodromies of singularities of the Hadamard and e\~ne product}


\author[R. P\'{e}rez-Marco]{Ricardo P\'{e}rez-Marco}

\address{Ricardo P\'erez-Marco\newline 
\indent  Institut de Math\'ematiques de Jussieu-Paris Rive Gauche, \newline
\indent CNRS, UMR 7586, \newline
\indent Universit\'e de Paris, B\^at. Sophie Germain, \newline 
\indent 75205 Paris, France}

\email{ricardo.perez-marco@gmail.com}

\begin{abstract}
\noindent
We prove that singularities with holomorphic monodromies are preserved by the Hadamard product. We find an explicit formula 
for the monodromy of the singularities, and similar formulas for the exponential e\~ne product. Using these formulas we 
get new direct proofs of classical results and the invariance of some rings of functions by Hadamard and e\~ne product 
(which is the product of the big Witt ring). This explains the good behavior of the e\~ne product on divisors.
\end{abstract}

\maketitle

%




%


\footnotesize{\noindent\textit{``...la nature du point singulier $\a\b$ ne d\'epend que de la nature 
des points singuliers $\a$ et $\b$...''}(\'Emile Borel on the Hadamard product, 1898)}

\smallskip



\begin{figure}[ht] \label{fig:monodromy2}
\centering
\resizebox{9cm}{!}{\includegraphics{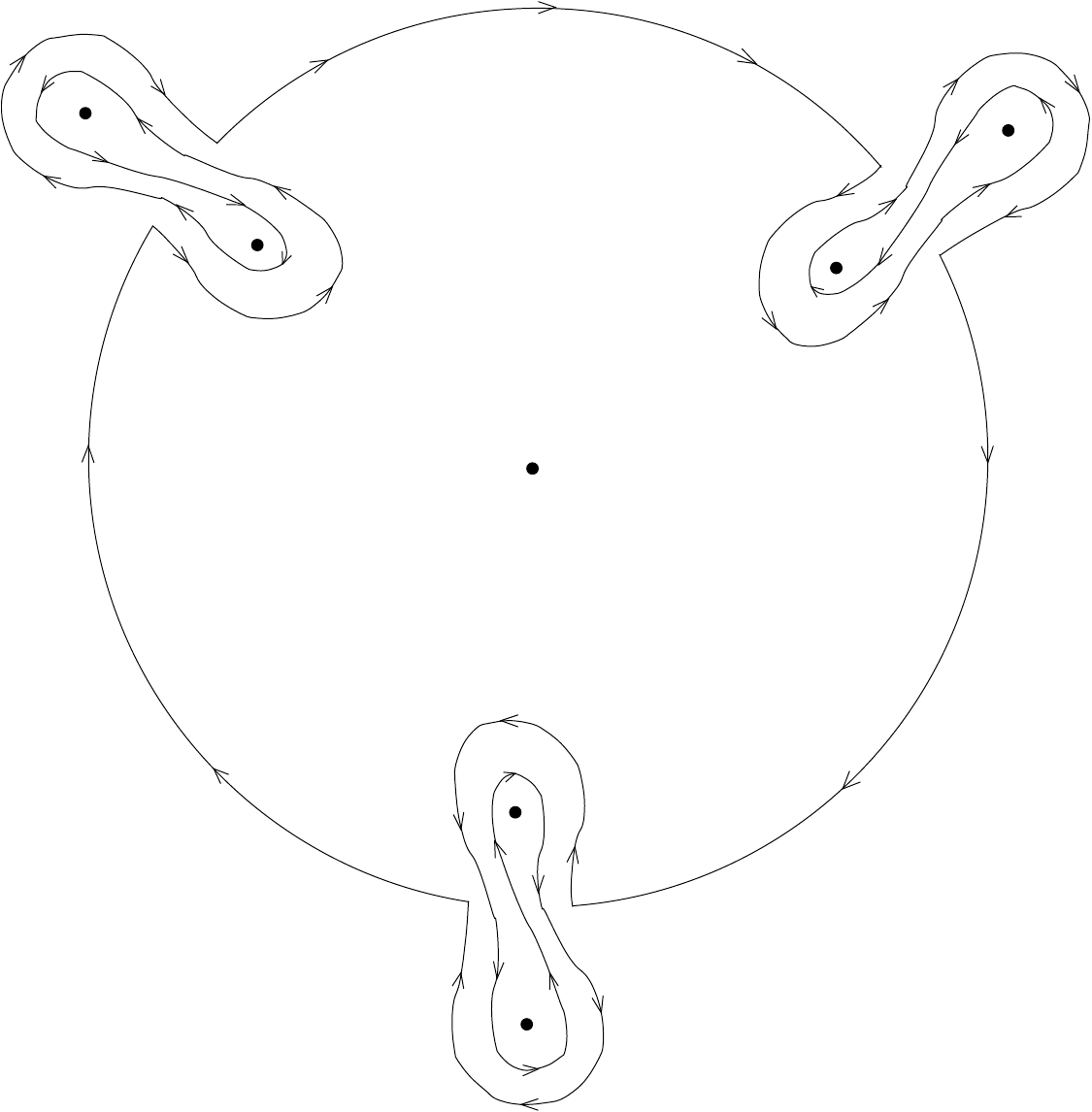}}    
\put (-215,245) {\scriptsize $\alpha_1$}
\put (-180,180) {\small $\frac{z_0}{\beta_1}$}
\put (-48,240) {\scriptsize $\alpha_2$}
\put (-90,180) {\small $\frac{z_0}{\beta_2}$}
\put (-165,20) {\scriptsize $\alpha_3$}
\put (-140,100) {\small $\frac{z_0}{\beta_3}$}
\put (-240,100) {\small $\hat\eta_{z_0}$}
\put (-125,150) {\small $0$}
\caption{\footnotesize{Choreographic monodromy integration contour.}} 
\end{figure}
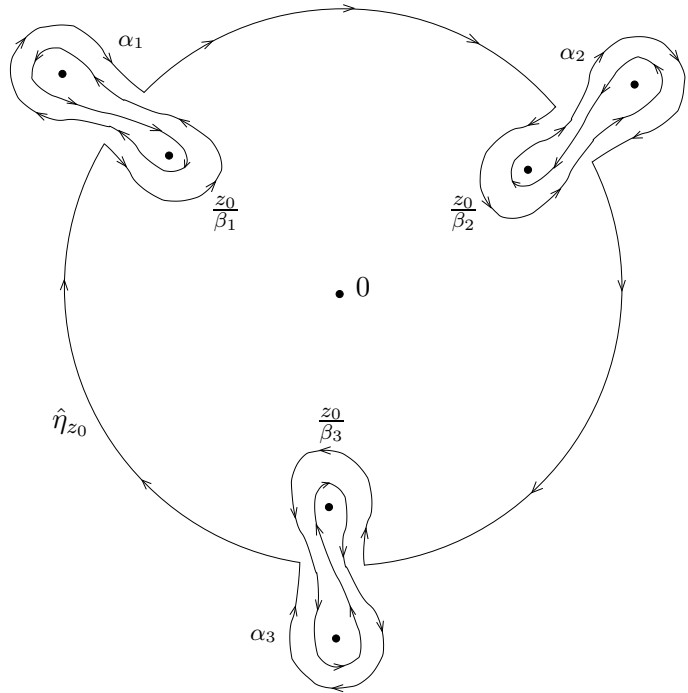

\medskip

\newpage

\section{Introduction and background.} \label{sec:introduction}

\subsection{Hadamard and e\~ne product.}

Given  two power series
\begin{align*}
F(z) &=A_0+A_1 z+A_2 z^2+\ldots =\sum_{n\geq 0} A_n \, z^n\\ 
G(z) &=B_0+B_1 z+B_2 z^2+\ldots =\sum_{n\geq 0} B_n \, z^n
\end{align*}
their classical \textit{Hadamard product}  is the power series
$$
F\odot G (z) =A_0B_0+A_1 B_1 z+A_2 B_2 z^2+\ldots =\sum_{n\geq 0} A_n B_n \, z^n
$$
and their \textit{exponential e\~ne product} is defined by
$$
F\star_e G (z) = -A_1 B_1 z -2 A_2 B_2 z^2+\ldots =-\sum_{n\geq 0} n A_n B_n \, z^n
$$
The Hadamard product and the exponential e\~ne product are commutative internal operations on the additive 
group of formal power series $\CC[[z]]$ (or $A[[z]]$ for a commutative ring $A$), and $(\CC[[z]], +,\odot)$
and $(\CC[[z]], +,\star_e)$ are commutative rings. These products are also internal operations on the additive 
subgroup $\CC\{\{z\}\}$ of power series with a positive radius of convergence, and $(\CC\{\{z\}\}, +,\odot)$
and $(\CC\{\{z\}\}, +,\star_e)$ are commutative subrings.

Hadamard (1899, \cite{Ha}) proved the \textit{Hadamard Multiplication Theorem} that locates the singularities of 
the principal branch of $F\odot G$ which are products of singularities of $F$ and $G$. 
The exponential e\~ne product has also a beautiful interpretation in terms of divisors, i.e. zeros and poles (2019, \cite{PM1}). More precisely,
if $a_0=b_0=0$ and we consider the exponential of the power series
\begin{align*}
f(z) &=\exp (F(z))= 1+a_1 z+a_2 z^2+\ldots =1+\sum_{n\geq  1} a_n \, z^n\\ 
g(z) &=\exp (G(z))=1+b_1 z+b_2 z^2+\ldots =1+\sum_{n\geq 1} b_n \, z^n
\end{align*}
and we define the e\~ne product as
$$
f\star g (z)=\exp (F\star_e G (z))=1+c_1 z+c_2 z^2+\ldots =1+\sum_{n\geq 1} c_n \, z^n
$$
then the coefficient $c_n$ is a universal polynomial with integer coefficients on 
$a_1, a_2, \ldots, a_n, b_1, b_2,\ldots, b_n$. Thus the exponential e\~ne product is the 
bilinearization of the e\~ne product through the exponential.
The universality of the expression of the coefficients $(c_n)$ allows to define 
in general the e\~ne product over 
a commutative ring of coefficients $A$, more precisely, if $\cA =1+A[[z]]$, then $\cA$ is a group for the multiplication
and when we adjoint the e\~ne product we obtain a commutative ring $(\cA, .,\star)$.

A brief digression is in order to explain that the ring $(\cA, .,\star)$ is 
the ``big Witt ring'' appearing in different forms in algebraic geometry, the theory of formal groups, etc It is difficult to 
find a proper reference for the big Witt ring. H. Lenstra has some nice unpublished notes \cite{Len} where he apologizes:

\medskip

\textit{``The literature (...) is in a somewhat unsatisfactory state: nobody seems to have any 
interest in Witt vectors beyond applying them for a purpose, and they are often treated in appendices 
to papers devoting to something else; also, the construction usually depends on a set of implicit or 
unintelligible formulae. Apparently, anybody who wishes to understand Witt vectors needs to construct them personally. 
That is now happening to myself.''}

\medskip

The author can only corroborate Lenstra's words, since the same occurred to him. The author has been studying this structure from the analytic point of view 
for more than 25 years, noticing in the last decade formal similarities with parts of the literature (like for example \cite{Haz}). Common mathematical wisdom 
tells us that important progress is achieved when fundamental mathematical objects are named. But the multiplicative operator in the big Witt ring didn't 
even had a proper name in the literature, which incidentally corroborates Lenstra's observation. 
Unaware of the big WItt rings, the author named it the ``e\~ne product'' which incidentally is nicer than 
``the multiplication of the big Witt ring''\footnote{The e\~ne product  seems to the author like a missing 
fundamental operation, as in many alphabets where the Spanish e\~ne letter is missing (\~n, pronounced ``ny'' as in ``canyon'').}. Our point of view 
focuses on the analytic properties of the e\~ne product (as in the present article), that are inexistent even today in the literature. For this reason we focus
on base rings $A$ that are of characteristic $0$ (in particular we don't need to take logarithmic derivatives to make sense of exponential formulas).
For an introduction and relevant analytic formulas from our point of view we refer to \cite{PM1} (2019). The e\~ne product is quickly introduced in an
elementary way through the action on polynomials, which we do next, closing this digression.

\medskip

Coming back to the previous notations, when $f$ and $g$ are non-constant polynomials 
(note in particular that in such case $F$ and $G$ are not polynomials), 
with respective 
roots $(\alpha)$ and $(\beta)$ (counted with multiplicity), so that 
\begin{align*}
f(z) &=\prod_{\a} \left (1-\frac{z}{\a}\right ) \\ 
g(z) &=\prod_{\b} \left (1-\frac{z}{\b}\right )
\end{align*}
then we have the remarkable formula
$$
f\star g (z)=\prod_{\a, \b } \left (1-\frac{z}{\a \b}\right ) \ .
$$
This interpretation with zeros can be taken as the starting point of the theory of 
the e\~ne product as is done in  \cite{PM1} (2019). Then the divisor property can be extended 
to entire and meromorphic functions on the plane, and even to those with transcendental singularities (2019, \cite{PM2}). 
The e\~ne product is both natural from the algebraic point of view, and from the analytic one. 
For example it respects Hadamard-Weierstrass factorizations (see section 9 of \cite{PM2}).
This property of zeros is closely related to the Hadamard Multiplication Theorem, although 
the Hadamard product lacks of a direct interpretation in terms of zeros. We have a neat formula
relating the exponential e\~ne product and the Hadamard product, namely
$$
F\star_e G =-K_0 \odot F \odot G
$$
where $K_0$ is the Koebe function 
$$
K_0(z)=-\sum_{n\geq 1} n \, z^n =-\frac{z}{(1-z)^2}
$$
Zeros and poles of a function $f$ correspond to singularities of $F$ with constant monodromies that are an integer multiple $n$ (the multiplicity) 
of $2 \pi i$·.

\medskip

For the definition of the e\~ne product and other algebraic and analytic properties and formulas 
we refer to \cite{PM1} where we extend 
the divisor interpretation to meromorphic functions with zeros and poles. In \cite{PM2}  we extend further the e\~ne product to the 
transalgebraic class. For the Riemann sphere $\PP^1\CC$, this transalgebraic class is composed by functions with a finite 
number of exponential singularities. These are of the form $R_0\exp R_1$ where $R_0$ and $R_1$ are rational functions. We prove in \cite{PM2} that 
the divisor interpretation still holds for exponential singularities, which is natural when we view these singularities \textit{\`a la Euler} 
as zeros or poles of ``infinite order''. From this point of view, we naturally introduce ``e\~ne poles''. 
Then functions with singularities with non-trivial monodromies arise naturally, 
in particular the hierarchy of polylogarithms. 

\medskip

Therefore, it becomes natural to investigate the extension of the divisor interpretation of the e\~ne product to 
singularities with non-trivial monodromy, i.e. non-uniform singularities in the XIX-th century terminology. A uniform 
transcendental singularity is by definition an isolated singularity without monodromy.

\medskip

Almost simultaneously to the discovery by Hadamard in 1898 (\cite{Ha}) of his Theorem determining the location of the 
singularities of $F\odot G$, \'Emile Borel (\cite{Bo}) proved that if $F$ and $G$ 
have uniform isolated singularities, i.e. isolated without monodromy, then the singularities of $F\odot G$ 
are also isolated and uniform. 
This result is related to the action of the e\~ne product on exponential singularities.
Borel also makes the vague, but on point, observation that the nature of 
the singularities of $F\odot G$ only depends on the nature of the 
singularities of $F$ and $G$, 

\medskip

\noindent\textit{``...la nature du point singulier $\a\b$ ne d\'epend que de la nature des points singuliers $\a$ et $\b$...''}\footnote{``... the nature of the singular 
point $\a\b$ depends only on the nature of the singular points $\a$ and $\b$...''} \\ (\'E. Borel, 1898) 

\medskip

The goal of this article is to make very precise this statement, by giving, in some natural situations, 
a new explicit formula for the monodromy of the singularities of  $F\odot G$ in terms of the 
monodromies of the singularities of $F$ and $G$.

\subsection{Holomorphic monodromy formulas.}

We need first some definitions and properties on monodromies. 
A singularity is \textit{isolated} if we can extend analytically the function in a small pointed neighborhood 
around the singularity by more than an angle of $2\pi$ without hitting other singularities.

\begin{definition}[Monodromy of an isolated singularity]
 
Let $F$ be an holomorphic function with an isolated singularity $\a\in \CC$. We denote 
$F_+$ the analytic continuation of $F$ when winding around $\a$ once in the positive 
orientation (in a neighborhood without any other singularity). The monodromy of $F$ 
at the point $\alpha\in \CC$ is
$$
\Delta_{\a} F =F_+(z)-F(z) \ .
$$
\end{definition}

\textbf{Example.} The simplest and more basic example of non-trivial monodromy is given by the logarithmic function, that 
we normalize properly, at the isolated singularity $\a=1$,
$$
\Delta_1 \left (\frac{1}{2\pi i} \log (z-1) \right ) =1
$$

\medskip

We consider in this article only holomorphic monodromies:

\begin{definition}[Holomorphic monodromy]
A function $F$ with an isolated singularity at $\a\in \CC$ has a holomorphic 
monodromy at $\a$ when $\Delta_{\a} F$ is holomorphic in a neighborhood of $\a$.
\end{definition}

Observe that if the monodromy is holomorphic then $\Delta_{\a}^2 F =0$. When $\Delta_{\a}^2 F =0$, if
$$
F_0=F-\frac{1}{2\pi i} \log(z-\a) \Delta_{\a} F 
$$
then  $\Delta_{\a} F_0 =0$, and $F_0$ has a uniform singularity at $\a$ (isolated without monodromy). 
We have proved:

\begin{proposition}\label{prop_decomposition}
When $\Delta_{\a}^2 F =0$, we can decompose $F$ uniquely as
$$
F=F_0+\frac{1}{2\pi i} \log(z-\a) \Delta_{\a} F
$$
where $F_0$ has a uniform singularity at $0$.
\end{proposition}
Note the minor abuse of notation since $F_0$ depends on the singularity $\a$. If needed we may use the notation 
$F_{\a, 0}$.

\begin{definition}[Totally holomorphic singularity]
The singularity $\a$ is totally holomorphic when both $\Delta_\a F$ and the germ $F_0$ are holomorphic at $\a$.
\end{definition}

Our main result computes the monodromy 
$\Delta_{\a\b} (F\odot G)$ from the monodromies $\Delta_{\a} F$ and $\Delta_{\b} G$ in the case of  
holomorphic singularities. We have a remarkable  explicit formula:

\begin{theorem}[Holomorphic monodromy formula for the Hadamard product]
\label{thm:Hadamard_holomorphic_monodromy}
We consider $F$ and $G$ holomorphic germs at $0$ with respective set of 
singularities $(\alpha)$ and $(\beta)$ in $\CC$. We 
assume that the singularities are isolated and holomorphic, that is, 
$\Delta_{\a} F$, resp.  $\Delta_{\b} G$, 
is holomorphic at $\a$, resp. $\b$.
Then the 
set of singularities of the principal branch of 
$F\odot G$ is contained in the product set $(\gamma)=(\alpha \beta)$ and is 
composed by isolated singularities which are holomorphic,
and we have the formula
\begin{align}\label{eq:general_monodromy_convolution}
\Delta_{\g} (F\odot G) (z)&= -\sum_{\substack{\a ,\b \\ \a\b=\g}} 
 \Res_{u=\a} \left ( \frac{F_0(u) \Delta_\b G(z/u)}{u} \right ) 
-\sum_{\substack{\a ,\b \\ \a\b=\g}} \Res_{u=\b} \left ( \frac{G_0(u) \Delta_\a F(z/u)}{u} \right )  \\ 
& \ \ \ \, - \frac{1}{2\pi i}  \sum_{\substack{\a ,\b \\ \a\b=\g}}  
\int_\a^{z/\b} \Delta_{\a} F(u) \Delta_{\b} G (z/u) \, \frac{du}{u} \nonumber
\end{align}
When the singularities are totally holomorphic, we have the simpler formula
\begin{equation}\label{eq:monodromy_convolution}
\Delta_{\g} (F\odot G) (z)= -\frac{1}{2\pi i}\sum_{\substack{\a ,\b \\ \a\b=\g}} 
\int_\a^{z/\b} \Delta_{\a} F(u) \Delta_{\b} G (z/u) \, \frac{du}{u} 
\end{equation}
\end{theorem}

The convolution in our formula appears also in the work of J. \'Ecalle in his study of algebras 
of resurgent functions (see for example formula 2.1.8 p.19 of \cite{Ec} where he uses the convolution 
for germs modulo holomorphic functions). There is certainly a relation that we don't fully understand 
at this point with \'Ecalle's theories.

 Notice the exceptional situation at $z=0$: it can happen that the 
monodromies of $F$ and $G$ are both holomorphic at $z=0$, but the monodromy of the Hadamard product has 
a singularity at $z=0$ with a non-trivial 
monodromy.  This monodromy is generated by the term $1/u$ in the integrand of the monodromy convolution
formula (\ref{eq:general_monodromy_convolution}). 
This occurs for instance for polylogarithms (see section \ref{sec:polylogarithms}).

\medskip

We observe that when  $\Delta_{\a} F =\Delta_{\b} G =0$ for all singularities $\a$ and $\b$, 
then the singularities are all holomorphic, 
 $\Delta_{\g} (F\odot G)=0$. Therefore, Borel's Theorem is a direct Corollary of our formula.

\medskip

Borel also observes that when $F$ is a rational function, then one can construct a differential 
operator $D_F$ such that the the singularities $\a\b$ of $F\odot G$ are ``of the same nature'' as those
of $D_F G$ at $\b$. The holomorphic monodromy formula makes this explicit when 
$G$ has holomorphic singularities in the more general situation of a meromorphic function $F$.

\begin{corollary}\label{cor_diff_ring}
Let $F$ be a meromorphic function in $\CC$ (for example a rational function), 
holomorphic at $0$, with set of poles $(\a)$, and $G$ with totally holomorphic 
singularities $(\b)$. The monodromies of $F\odot G$ 
are in the differential ring generated by the $\Delta_\b G(z/\a)$, with field of constants generated by the 
coefficients of the polar parts of $F$. 
More precisely, consider the polar part of $F$ 
at each pole $\a$
$$
\sum_{k=1}^d \frac{a_{k,\a}}{(u-\a)^k} 
$$
then we have
$$
\Delta_\g (F\odot G) =-\sum_{\substack{\a ,\b , k\\ \a\b=\g}} \frac{ a_{k,\a}}{ (k-1)!} 
 \left [\frac{d^{k-1}}{du^k}  \left (\frac{\Delta_\beta G (z/u)}{u} \right )\right ]_{u=\a} \ \ .
$$
\end{corollary}

\textbf{Example.}
A simple example occurs when we take $F=-K_0$, where $K_0(z)=z/(1-z)^2$ is the Koebe function 
which has a simple pole of order $2$ at $z=1$ and polar part
$$
-K_0(z)=-\frac{1}{z-1}-\frac{1}{(z-1)^2} \ .
$$
Then we compute
\begin{align*}
\Delta_\b (-K_0\odot G) &=-\Res_{u=1} \left (\left (-\frac{1}{u-1}-\frac{1}{(u-1)^2}\right ) \frac{\Delta_\b G(z/u)}{u}  \right ) \\ 
& = [\Delta_\b G(z/u)/u]_{u=1}   + \left [\frac{\Delta_\b G'(z/u)}{u} -
\frac{\Delta_\b G(z/u)}{u^2}\right ]_{u=1} \\
&= \Delta_\b G(z) + \Delta_\b G'(z) -\Delta_\b G(z) \\
&=\Delta_\b G'(z)\\
&=\left (\Delta_\b G\right )'(z)
\end{align*}

\begin{proposition}
The monodromies of the singularities of the 
Hadamard product with the negative of the Koebe function are the derivatives of the monodromies.
\end{proposition}

\medskip

We have a similar formula for the monodromies of the exponential e\~ne product.

\begin{theorem}[Holomorphic monodromy formula for the exponential e\~ne product]\label{thm:holomorphic}
We consider $F$ and $G$ holomorphic germs at $0$ with respective sets of singularities $(\alpha)$ and $(\beta)$. We 
assume that the singularities are isolated and holomorphic. Then the set of singularities of 
the principal branch of $F\star_e G$ is contained in the product set
$(\gamma)=(\alpha \beta)$ and is 
composed by isolated singularities with holomorphic monodromies, and we have
\begin{align}\label{eq:ene_monodromy_convolution}
\Delta_{\g} (F\star_e G) (z)&= \sum_{\substack{\a ,\b \\ \a\b=\g}} 
 \Res_{u=\a} \left ( F'_0(u) \Delta_\b G(z/u) \right ) 
+\sum_{\substack{\a ,\b \\ \a\b=\g}} \Res_{u=\b} \left ( G_0(u) \Delta_\a F'(z/u) \right )  \\
& \ \ \ \, + \frac{1}{2\pi i} \sum_{\substack{\a ,\b \\ \a\b=\g}}  \Delta_\a F(\a) \Delta_\b G(z/ \a)  
+ \frac{1}{2\pi i}  \sum_{\substack{\a ,\b \\ \a\b=\g}}  
\int_\a^{z/\b} \Delta_{\a} F'(u) \Delta_\b G (z/u) \, du \nonumber
\end{align}
When the singularities are totally holomorphic, we have the simpler formula
\begin{equation}\label{eq:ene_tot_hol_monodromy_convolution}
\Delta_{\g} (F\star_e G) (z)= \frac{1}{2\pi i} \sum_{\substack{\a ,\b \\ \a\b=\g}}  \Delta_\a F(\a) \Delta_\b G(z/ \a)  
+ \frac{1}{2\pi i}  \sum_{\substack{\a ,\b \\ \a\b=\g}}  
\int_\a^{z/\b} \Delta_{\a} F'(u) \Delta_\b G (z/u) \, du 
\end{equation}

\end{theorem}

Notice this time the absence of the factor $1/u$ in the integral of the monodromy convolution formula 
(\ref{eq:ene_monodromy_convolution}) for the exponential e\~ne product generates no extra singularities 
at $z=0$ for non-principal branches of $F\star_e G$. This explains why the monodromies 
of singularities of the 
exponential e\~ne product have a better analytic behavior than those for the Hadamard product. 
The symmetry of the formula on $F$ and $G$ is clear in the first line, and for the second line it 
follows by integration by parts using basic 
properties of the monodromy operator $\Delta_\a$ (see Proposition \ref{prop:symmetry_ene}).

For the exponential e\~ne product we have the same Borel's type of Theorem. It follows from formula 
(\ref{eq:ene_monodromy_convolution}) that if $\Delta_\a F=\Delta_\b G =0$ then $\Delta_\g (F\star_e G)=0$, hence we have:
\begin{corollary}
 If $F$ and $G$ have only uniform singularities
then $F\star_e G$ has only uniform singularities, i.e. 
$\Delta_{\g} (F\star_e G)=0$.
\end{corollary}

We have also an analogue of Corollary \ref{cor_diff_ring}. The result is stronger because of the 
better analytic properties of the e\~ne product.

\begin{corollary}\label{cor:ene_improved_borel}
Let $F$ be a function in $\CC$ with a discrete set $(\a)$ of singularities with constant monodromies, such that 
the germs $F_0$ are meromorphic, $F$ is
holomorphic at $0$, and $G$ with totally holomorphic 
singularities $(\b)$. The monodromies of $F\star_e G$ 
are in the differential ring generated by the $\Delta_\b G(z/\a)$, with field of constants generated by the 
coefficients of the polar parts of $F$ and the constants $\left (\frac{\Delta_\a F}{2\pi i} \right )$. 
More precisely, consider the polar part of each $F_0$ 
at each pole $\a$
$$
\sum_{k=1}^d \frac{a_{k,\a}}{(u-\a)^k} \ ,
$$
then we have
$$
\Delta_\g (F\odot G) =-\sum_{\substack{\a ,\b , k\\ \a\b=\g}} \frac{a_{k,\a}}{(k-1)!} 
 \left [\frac{d^{k}}{du^k}  \left (\Delta_\beta G (z/u) \right )\right ]_{u=\a}
 +\sum_{\substack{\a ,\b \\ \a\b=\g}}  \frac{\Delta_\a F(\a)}{2\pi i} \, \Delta_\b G(z/ \a) \ \ .
$$
\end{corollary}

A Corollary of the algebraic nature of these  
monodromy formulas is the invariance by the Hadamard and e\~ne product of some 
natural rings of functions. Consider a field $K\subset \CC$ with $2\pi i \in K$, and consider the ring, for the usual addition and multiplication,  
$PML(K)$ (Polynomial Logarithmic Monodromy ring with coefficients in $K$)
of germs holomorphic at $0$ with only isolated singularities located at points in $K$, 
and  with monodromy in $K[z,\log z]$. 

\begin{corollary}
 The $PLM(K)$ ring is closed under Hadamard, resp. e\~ne, product,  and is the 
 minimal Hadamard ring, resp. e\~ne ring, containing the class of 
 functions with polynomial monodromies in $K[z]$ with isolated singularities located at
 points of $K$.
\end{corollary}

The author doesn't know of a direct proof using the definition of the Hadamard product. It is probably a difficult 
combinatorial problem, but will be interesting to have one.

The monodromy formulas presented in this article do extend to more general singularities. 
We refer to \cite{PM3} for these more general formulas and other new applications.

\section{Convolution formulas.}

The main tool in the proof of Hadamard Multiplication Theorem is Pincherle (or Hadamard) convolution formula (1885, see 
\cite{Pin}, \cite{Lau} and \cite{Ha}):
\begin{proposition}[Pincherle convolution formula]
The Hadamard product has the integral form
\begin{equation} \label{convolution}
F\odot G (z)=\frac{1}{2\pi i} \int_\eta F(u)G(z/u)\, \frac{du}{u}
\end{equation}
where $\eta$ is a positively oriented circle centered at $0$ of radius $r>0$ with $|z|/R_G <r< R_F$, 
where $R_F$, resp. $R_G$, is the radius of convergence of $F$, resp. $G$,
so that $F(u)$ and $G(z/u)$ are well defined.
\end{proposition}

We can take for  $\eta$ any Jordan curve located  
in the annulus bounded by the circles of radii
$R_F$ and $|z|/R_G$ and  with winding number $+1$ with respect to $0$. 

\begin{proof}
The convolution formula immediately follows from the integration term by term of the series  
$$
\frac{F(u)G(z/u)}{u} =\sum_{n,m\geq 0} F_n G_m u^{n-m-1} z^m
$$
and the application of Cauchy formula
$$
\frac{1}{2\pi i} \int_\eta u^{n-m-1} \, du =\delta_{n,m} 
$$
where $\delta_{n,m}$ denotes the Kronecker symbol. 
\end{proof}

For the exponential e\~ne product we have a similar convolution formula:

\begin{proposition}[Exponential e\~ne product convolution formula]
The exponential e\~ne product has the integral form
\begin{equation} \label{convolution_ene}
F\star_e G (z)=-\frac{1}{2\pi i} \int_\eta F'(u)G(z/u)\, du
\end{equation}
where $\eta$ is a circle centered at $0$ with the same conditions as before.
\end{proposition}

\begin{proof}
The proof is similar observing that $R_{F'}=R_F$ and integrating term by term
$$
F'(u)G(z/u) =\sum_{n,m\geq 0} n F_n G_m u^{n-m-1} z^m
$$
and using Cauchy formula as before. 
\end{proof}

The convolution formula (\ref{convolution}) gives the analytic continuation of 
the Hadamard product $F\odot G$ using the analytic continuation of 
$F$ and $G$. We only need to deform homotopically the contour $\eta$ when we move $z$ around. The 
only obstruction to this continuation follows from 
a close inspection of the convolution formula: $F\odot G (z)$ does extend
analytically unless $z=0$ (except for the principal branch) or when we hit a point $z\in \CC$ such that we have $u\in \CC$ fulfilling both conditions
\begin{equation*}
\begin{cases}
u&=\a \\
z/u&=\b
\end{cases}
\end{equation*}
with $\a$ and $\b$ singularities of $F$ and $G$ respectively. This happens only when  $z=\a \b$.
Thus we have proved Hadamard Theorem:

\begin{theorem}[Hadamard Multiplication Theorem, 1899, \cite{Ha}]
 The singularities of the principal branch of $F\odot G$ are of the form $\g=\a \b$ where 
 $\a$ and $\b$ are singularities of $F$ and $G$ respectively.
\end{theorem}

The origin $0$ is not a singularity of the principal branch by assumption. But
the convolution formula  shows that the origin $0$ can become a singularity of other 
branches of the analytic continuation of $F\odot G$ because of the $1/u$ factor in the integrand.
Moreover, the convolution formula also proves that if $F$ and $G$ are fluent in the sense of 
Liouville and Ritt, which, roughly speaking, means  that the functions have 
an analytic extension around singularities, then 
their Hadamard product $F\odot G$ is fluent. 
The definition of fluency given by Ritt in his book\footnote{Ritt's book is from 
the pre-differential algebra era.} (1948,\cite{Ri}) on Liouville 
theory of integration on finite terms (1833, \cite{Li1}, \cite{Li2}) is not 
precise and indeed there exists 
more or less stronger versions of fluency (see section 5.4.1 of \cite{Kh}, and in particular \cite{Le}). Fluency is a key property 
of functions in the old Liouville classification of transcendental functions. In simpler terms, we can 
state the following Corollary of Hadamard Multiplication Theorem (or the Convolution Formula),

\begin{corollary}
If both $F$ and $G$ have only isolated singularities with monodromy, then $F\odot G$ has only isolated 
singularities with monodromy.
\end{corollary}

These considerations give the following geometric improvement of Hadamard Theorem 
(the reader, if not familiar, can skip this Theorem where we use the language of 
log-Riemann surfaces, see \cite{BPM1} for general background, \cite{BPM2} for more 
general definitions, and \cite{BPM3} and \cite{BPM4} for further properties).

\begin{theorem}
Let $\cS_F$ and $\cS_G$  be the log-Riemann surfaces of the 
germs at $0$ defined by $F$ and $G$ respectively. Then the log-Riemann surface $\cS_{F\odot G}$ has a ramification set $\cR_{F\odot G}$ such that 
$$
\pi_{F\odot G} (\cR_{F\odot G}) \subset \{0\} \cup (\pi_F (\cR_{F})  . \pi_G (\cR_{G}))  
$$
where $\pi_{F\odot G} : \cS_{F\odot G} \to \CC$, $\pi_{F} : \cS_{F} \to \CC$ and 
$\pi_{G} : \cS_{G} \to \CC$ are the canonical projections, and $A.B$ denotes the set of all products $ab$ with $a\in A$ and $b\in B$.
In particular, if  the ramification sets $\cR_F$ and $\cR_G$ 
are discrete, then $\cR_{F\odot G}$ is discrete. 
\end{theorem}

\medskip

For the exponential e\~ne product we have a similar Theorem as Hadamard Multiplication Theorem using the 
formula $F\star_e G=-K_0\odot F\odot G$ (or using the convolution formula with the same proof
observing that that the singularities of $F$ and $F'$ are the same). Hence we obtain:

\begin{theorem}[Singularities of the exponential e\~ne product]
The singularities of the principal branch of $F\star_e G$ are of the form $\g=\a \b$ where 
 $\a$ and $\b$ are singularities of $F$ and $G$ respectively.
\end{theorem}

We have a corresponding Theorem for the log-Riemann surface $\cS_{F\star_e G}$ which is 
simpler since we don't need to add the origin in the locus of the projection.

\begin{theorem}
Let $\cS_F$ and $\cS_G$  be the log-Riemann surfaces of the 
germs at $0$ defined by $F$ and $G$ respectively. Then the log-Riemann surface $\cS_{F\star_e G}$ has a ramification set $\cR_{F\star_e G}$ such that 
$$
\pi_{F\star_e G} (\cR_{F\star_e G}) \subset \pi_F (\cR_{F})  . \pi_G (\cR_{G})  
$$
where $\pi_{F\star_e G} : \cS_{F\star_e G} \to \CC$, $\pi_{F} : \cS_{F} \to \CC$ and 
$\pi_{G} : \cS_{G} \to \CC$ are the canonical projections, and $A.B$ denotes the set of all products $ab$ with $a\in A$ and $b\in B$.
In particular, if  the ramification sets $\cR_F$ and $\cR_G$ 
are discrete, then $\cR_{F\star_e G}$ is discrete. 
\end{theorem}

\section{Monodromy of singularities and monodromy operator.}

As already observed, the inspection of the convolution formula shows also that if both 
$F$ and $G$ have isolated singularities 
with monodromy  then the singularities of $F\odot G$ and of $F\star_e G$ are also isolated with 
monodromy (including the case of trivial monodromy). More precisely, we recall again the definition of  
the monodromy and we define the monodromy operator:

\begin{definition}[Monodromy and operator monodromy of an isolated singularity]
 
Let $F$ be an holomorphic function with an isolated singularity $\a\in \CC$. We denote 
$F_+$ the analytic continuation of $F$ when turning  around $\a$ once in the positive 
orientation (in a neighborhood without any other singularity). The monodromy of $F$ 
at the point $\alpha\in \CC$ is
$$
\Delta_{\a} F =F_+(z)-F(z) \ .
$$
This definition can be extended also to regular points $\a$ of $F$. The map $\Delta_\a$ 
defines a linear operator, the monodromy 
operator at $\a$,
on the vector space $V_\a$ of holomorphic functions having a regular point or an 
isolated singularity at $\a$.  We also 
define the operator $\Sigma_\a$ such that
$$
\Sigma_\a F=F_+
$$
and we have
\begin{equation*}
\Sigma_\a=I+\Delta_\a
\end{equation*}
In general, we define that monodromy operator along a path $\gamma$, $\sigma_\gamma : (\CC, \gamma(0))\to (\CC, \gamma(1))$, which
associates to a holomorphic germ at $\gamma(0)$ its Weierstrass holomorphic continuation along $\gamma$ at $\gamma(1)$. In a domain 
where the germ is holomorphic, we have  that $\sigma_\gamma =\sigma_{[\gamma]}$ only depends on the homotopy class 
$[\g]$ of $\gamma$, and for a loop 
$\gamma$ with winding number $1$ with respect to $\a$, we have
$\Sigma_\alpha =\sigma_{[\gamma]}$.
\end{definition}

Note that $\Delta_{\a} F$ can develop non-local singularities 
elsewhere at other points distinct from $\a$. The structure of the monodromy at $\alpha$ is 
important, and justifies the following definitions.

\begin{definition}
A function $F$ with an isolated singularity at $\a\in \CC$ has a holomorphic, 
resp. meromorphic,  uniform, monodromy at $\a$ when $\Delta_{\a} F$ 
is holomorphic, resp. meromorphic, with a uniform isolated singularity at $\a$, in a neighborhood of $\a$.
The singularity $\a$ is totally holomorphic if both $\Delta_\a F$ and $F_0=F-\frac{1}{2\pi i} \log (z-\a) \Delta_\a F$ 
are holomorphic.
\end{definition}

Recall (see Proposition \ref{prop_decomposition}) that when $\Delta_\a^2 F=0$ we 
can write uniquely in a neighborhood of $\a$
$$
F=F_0+\frac{1}{2\pi i} \log (z-\a) \Delta_\a F
$$
where $F_0$ has a uniform isolated singularity at $\a$.
The growth behavior of $F$ near the singularity $\a$ is important.

\begin{definition}\label{def_integrable}
Let $F$ be function with a regular point or an isolated singularity at $\a\in \CC$. Then $F$ is 
integrable at $\a$, or $\a$ is an integrable singularity, if we have   
$$
\Slim_{z\to \a} (z-\a) \,  F (z) \to 0
$$ 
where this is a Stolz limit, that is, $z\to \a$ with $\arg (z-\a)$ bounded. The singularity $\a$ 
has order $\rho >0$, if $ F(z)=\cO(|z-\a|^{-\rho})$  in any Stolz angle\footnote{These are also called 
``singularities with moderate growth in sectors'', see \cite{Sau}, section 9.2, or ``regular singularities'' in the context 
of solutions of differential systems, see \cite{AB}, p.8.}.
\end{definition}


The logarithmic function branched at $\a$, $\log(z-\a)$, is integrable at $\a$.
The space of functions that are integrable, of order $\rho$ or finite 
order functions at $\a$ form a $\CC$-vector space. These vector spaces are $\cO_\a$-modules, 
where $\cO_\a$ is the local ring of holomorphic germs at $\a$. 
For an integrable singularity we have
$$
\int_{[\alpha,\alpha +\eps]} F (u)  \, du \to 0
$$
and 
$$
\int_{\eta}  F (u)  \, du \to 0
$$
when $\eps\in \CC$ is small and $\eps\to 0$, and 
the loop $\eta$ is a small circle around $\alpha$ with radius $\to 0$. 
Adapting the arguments given in \cite{Ju} we can prove that the Hadamard product of integrable singularities 
is integrable, but we don't need this result in this article.

\begin{proposition}\label{prop:holomorphic+integrable}
If $F$ is holomorphic in a pointed neighborhood of $\a$, i.e. $\Delta_\a F=0$, and integrable, 
then $F$ is holomorphic at $\a$. 
\end{proposition}

\begin{proof}
We can write a converging  Laurent expansion in a pointed neighborhood of $\a$,
$$
F(z)=\sum_{n\in \ZZ} a_n (z-\a )^n
$$
and we have the Cauchy formula for the coefficients
$$
a_n=\frac{1}{2\pi i} \int_\eta \frac{f(u)}{(z-u)^n} \, du \ .
$$
When we shrink $\eta$ to $\a$, the integrability estimates show that $a_n=0$ for all $n<0$, hence $F$ is holomorphic.
\end{proof}

\begin{proposition}\label{prop:product}
 We have
 $$
 \Delta_\a (F.G)= \Delta_\a (F).G + F. \Delta_\a (G) +\Delta_\a (F).\Delta_\a (G)
 $$
\end{proposition}

\begin{proof}
It follows by analytic continuation that 
$$
\Sigma_\a (F.G) = \Sigma_\a (F).\Sigma_\a (G)
$$
and using $\Sigma_\a =I+\Delta_\a$ gives the result.
\end{proof}

\begin{corollary} \label{cor:representation}
If $\Delta_\a^2 (F)=0$, in particular when the monodromy of $F$ at $\a$ is holomorphic or meromorphic, then 
$$
\Delta_\a \Big (F-\frac{1}{2\pi i} \log(z-\a) \, \Delta_\a (F) \Big ) =0
$$
In that case we can write,
$$
F(z)=F_0(z) + \frac{1}{2\pi i} \log(z-\a) \, \Delta_\a (F)
$$
where $\Delta_\a F_0=0$. 
\end{corollary}

\begin{proof}
We observe  that 
$$
\Delta_\a \left ( \frac{1}{2\pi i} \log(z-\a)\right ) =1
$$
and then from the previous Proposition we get
\begin{align*}
\Delta_\a \Big (F-\frac{1}{2\pi i} \log(z-\a) \, \Delta_\a (F) \Big ) &= \Delta_\a (F)- 1.\Delta_\a (F)- \frac{1}{2\pi i} \log(z-\a) . \Delta_\a^2 (F) - 1. \Delta_\a^2 (F)\\
&=\Delta_\a (F) -\Delta_\a (F)-0-0\\
&=0
\end{align*}

\end{proof}

\begin{corollary}
We assume that $\Delta_\a^2 F=0$. The singularity $\a$ of $F$ is integrable if 
and only if the singularity is totally holomorphic, that is,
$\Delta_\a F$  and $F_0$ are both holomorphic at $\a$. 
\end{corollary}

\begin{proof}
The condition is necessary. 
If $F$ is integrable at $\a$, directly from the definition we get that $F_+$ is integrable, hence $\Delta_\a F=F_+-F$ is integrable and holomorphic in a pointed 
neighborhood of $\a$. Proposition \ref{prop:holomorphic+integrable} implies that $\Delta_\a F$ is a holomorphic function at $\a$. Then also 
$\frac{1}{2\pi i} \log(z-\a) \, \Delta_\a (F)$ is integrable at $\a$, hence $F_0(z) =F(z)-\frac{1}{2\pi i} \log(z-\a) \, \Delta_\a (F)$ is also integrable.
Again, by Proposition \ref{prop:holomorphic+integrable}, $F_0$ being holomorphic in a pointed neighborhood and integrable, it is holomorphic.

The condition is sufficient. We assume that $F_0$ and $\Delta_\a F$ are both 
holomorphic at $\a$. We have that $\frac{1}{2\pi i} \log(z-\a) \, \Delta_\a (F)$ is integrable because
$r\log r $ is integrable on $\RR_+$ at $r=0$. Also, $F_0$ is integrable since is it holomorphic at $\a$. Now, adding these two integrable functions it follows that $F$ is integrable.
\end{proof}

\medskip

Now, observe that we have $F_+=F+\Delta_\a F$, hence
$$
(F_+)'=F'+(\Delta_\a F)'
$$
and since $(F')_+ =(F_+)'$, which is obvious by analytic continuation, then we conclude that the monodromy  operator $\Delta_\alpha$ 
commutes with the derivation.

\begin{proposition}\label{prop:commuting}
We have 
$$
\Delta_\a (F')=(\Delta_\a F)'
$$
\end{proposition}

\begin{corollary}
Let $\vartheta = P(D)=\sum_{n=0}^N a_n D^n$ with $D=d/dz$ be a differential operator with coefficients
$(a_n)$ with isolated singularities without monodromy. We have
$$
\Delta_\a (\vartheta F)=\vartheta (\Delta_\a F) \ .
$$
Therefore, if $F$ satisfies the differential equation $\vartheta F =0$ then the monodromy $\Delta_\a F$ satisfies the same differential 
equation
$$
\vartheta (\Delta_\a F) =0 
$$
and $\Delta_\a F$ belongs to the finite dimensional vector space of solutions. We have the same result for a system of differential equations. 
\end{corollary}

\begin{proof}
Using Proposition \ref{prop:product} we have that 
$$
\Delta_\a (a_n F^{(n)})=a_n \Delta_\a (F^{(n)})=a_n \Delta_\a (F)^{(n)}
$$
and the result follows.
\end{proof}

The monodromy $\Delta_\a F$ can have an isolated singularity at $\a$ with non-trivial monodromy. For example this happens at singular 
points of differential equations, and of course for algebraic functions. For these we have the simple, but important, observation:

\begin{proposition}\label{prop:algebraic}
We assume that $F$ is an algebraic function satisfying the algebraic equation
$$
P(z, F)=0
$$
with $P\in \CC[x,y]$. Then $\Delta_\a F$ is an algebraic function satisfying an algebraic equation of lower degree
$$
Q(z, \Delta_\a F)=0
$$
with $Q\in \CC[x,y]$ and $\deg_y Q < \deg_y P$
\end{proposition}

\begin{proof}
By analytic continuation around $\a$ we have
$$
P(z, F+\Delta_\a F)=0
$$
developing the powers of $F+\Delta_\a F$, using $P(z, F)=0$, and factoring out $\Delta_\a F$ we get an explicit expression for $Q$ that has a lower degree in the second variable.
\end{proof}

 We will use the following change of variables formula.

\begin{proposition}[Change of variables formula]\label{prop:change_variables}
Let $\varphi: (\CC, \beta)\to (\CC, \alpha)$ a local holomorphic diffeomorphism, $\alpha =\varphi(\beta)$. We have
$$
\Delta_{\varphi^{-1}(\alpha)} (F\circ \varphi)= \Delta_{\alpha}(F) \circ \varphi \ .
$$
\end{proposition}

\begin{proof}
We consider a local loop $\gamma$ enclosing $\a$ with winding number $1$ with respect to $\a$. 
Its pre-image $\varphi^{-1}(\gamma)$ has winding number $1$ with respect to $\beta$. When we continue
analytically along $\varphi^{-1}(\gamma )$, the 
return map for $F\circ \varphi$ is $F\circ \varphi + \Delta_\beta (F\circ \varphi)$ by definition of the monodromy. The return 
analytic continuation of $F$ along $\gamma$ is $F+\Delta_{\a}F$, hence we have
$$
F\circ \varphi + \Delta_\beta (F\circ \varphi)= \left (F+\Delta_{\a}(F)\right )\circ \varphi = F\circ \varphi + \Delta_{\a} (F)\circ \varphi 
$$
and we get the stated formula.
\end{proof}

\begin{corollary}\label{cor:change_1/x}
Let $z_0\in \CC^*$. We have 
$$
\Delta_\beta (F(z_0/z))= \Delta_{z_0/\beta}(F) (z_0/z) \ .
$$
\end{corollary}

\begin{proof}
With the change of variables  $\varphi(z)=z_0/z$, thus $\varphi^{-1}(z)=z_0/z$, we get the result using Proposition  \ref{prop:change_variables}.
\end{proof}

The formula in this Corollary remains valid for the monodromy at $\infty$ in the Riemann sphere, $\alpha =\infty \in \overline{\CC}$. 
Thus, the formula holds in general for a Mo\"ebius transformation $\varphi$.

\medskip

This result is a particular case for $n=1$ (local diffeomorphism) of the local degree $n\geq 1$ case (we 
will not use this more general result).

\begin{proposition}\label{prop:degree_n}
Let $\varphi: (\CC, \beta)\to (\CC, \alpha)$ a local holomorphic map with $\beta$ a critical point of degree $n\geq1$ 
(hence $\varphi$ is of local degree $n$). We have
$$
\Delta_{\alpha}(F) \circ \varphi = \sum_{k=1}^n \binom{n}{k} \Delta_{\beta}^k (F\circ \varphi)   
$$
\end{proposition}

\begin{lemma}\label{lem:degree_n}
Let $\gamma$ be a local loop with winding number  $n\geq 1$ with respect to $\beta\in \CC$ in a neighborhood where $\beta$ is the only singularity of $F$
in this neighborhood. 
We denote $\Delta_\beta^{(n)}F$ the monodromy of $F$ along $\gamma$ that only depends on the homotopy class $[\gamma]$, i.e. on its winding number $n\geq 1$. 
We have
$$
\Delta_\beta^{(n)} F =\sum_{k=1}^{n} \binom{n}{k} \Delta_\beta^k F \ . 
$$
\end{lemma}

\begin{proof}We observe that 
$$
\Delta_\beta^{(n)} = \sigma_{n.[\gamma]}^n-I =\Sigma^n_\beta -I \ .
$$
Since $\Sigma_\beta =I+\Delta_\beta$, the result follows from Newton binomial formula.
\end{proof}

\begin{proof}[Proof of Proposition \ref{prop:degree_n}]
We carry out the same proof as before, but this time the pre-image $\varphi^{-1}(\gamma)$ has 
winding number $n\geq 1$ with respect to $\beta=\varphi^{-1}(\a)$. When we continue analytically 
along $\gamma$ the 
return map of $F$ is $F + \Delta_\alpha (F)$ by definition of the monodromy. The return 
analytic continuation of $F\circ \varphi$ along $\varphi^{-1}(\gamma)$ is 
$F\circ \varphi +\Delta^{(n)}_{\beta} (F\circ \varphi)$ 
hence we have
$$
F\circ \varphi + \Delta_\alpha (F)\circ \varphi= F\circ \varphi + \Delta_{\beta}^{(n)} (F\circ \varphi ) 
$$
so $\Delta_\alpha (F)\circ \varphi=  \Delta_{\beta}^{(n)} (F\circ \varphi)$
and the formula follows from Lemma \ref{lem:degree_n}.
\end{proof}

We leave the general study of monodromies $\Delta_\a F$ having an isolated singularity at $\a$ with non-trivial monodromy to future articles.

\section{Monodromy formula for integrable singularities.}

As a preparation for the general case, we first prove 
the monodromy formula (\ref{eq:monodromy_convolution}) in the 
integrable case. Thanks to the integrability conditions, there is no local contributions in the 
integration path argument that is central in the proof. In the next section we treat the general case.

\subsection{Proof of the totally holomorphic monodromy formula: Single singularity case} \label{sec:proof}

We prove in this section Theorems \ref{thm:Hadamard_holomorphic_monodromy} and \ref{thm:holomorphic}
in the integrable case.
The proof of Theorem \ref{thm:holomorphic} is similar and indications will be given at the end, so we concentrate in 
the proof of \ref{thm:Hadamard_holomorphic_monodromy}.

The first observation is that the result is purely local near the singularities. We may have some multiplicity when there exists
distinct pairs $(\a, \b)\not= (\a',\b')$ such that $\a\b=\a'\b'$. Then we have a linear superposition 
of the different contributions. This shows that we can be more precise in Theorems \ref{thm:Hadamard_holomorphic_monodromy} since  
there is no need to assume in the Theorem that all singularities have the same 
holomorphic structure but only those $\a$'s and $\b$'s that contribute to the location $\g=\a\b$.

We consider first the case where there is no multiplicity, that is 
$\g=\a \b$ for a unique pair $(\a, \b)$. The general case is a superposition of this case. 
We can assume that $\a =\b =\g =1$, and the radii of convergence are $R_F=R_G=1$. 

We consider a current value $z$ close to $1$, and we follow the analytic extension of $F\odot G$ when 
$z$, starting at $z_0$, turns around $1$ once in the positive direction. We can start at $z_0$  with $|z_0|<1$, and
$z_0$ close to $1$.


\bigskip

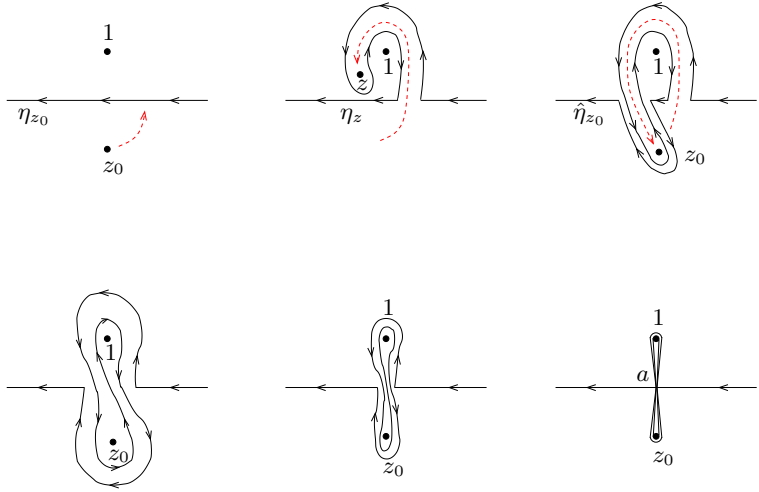
\begin{figure}[h] \label{fig:onemonodromy1}
\centering
\resizebox{10cm}{!}{\includegraphics{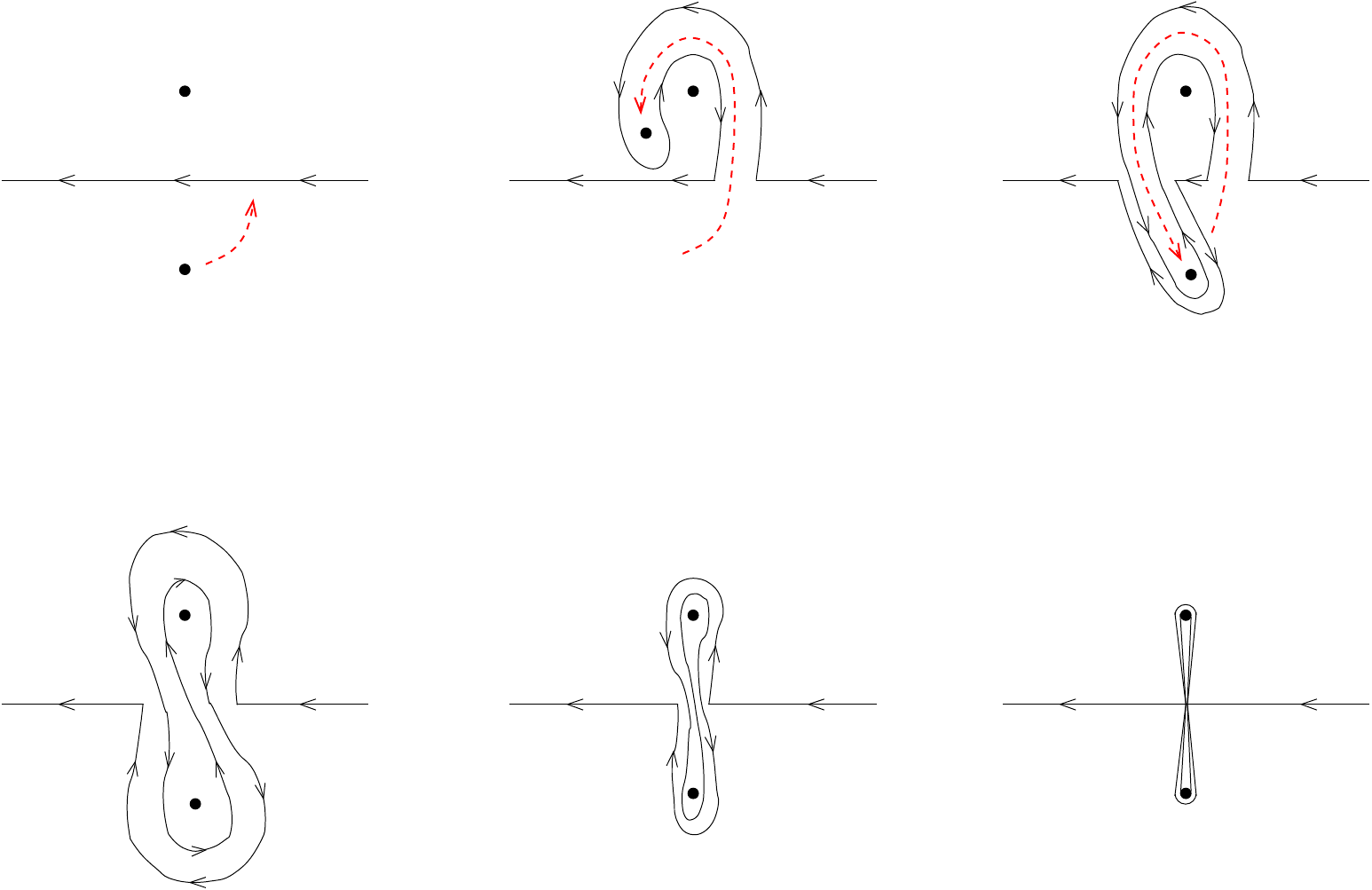}}    
\put (-280,140) {\scriptsize $\eta_{z_0}$}
\put (-248,170) {\scriptsize $1$}
\put (-248,120) {\scriptsize $z_0$}
\put (-158,140) {\scriptsize $\eta_z$}
\put (-142,157) {\scriptsize $1$}
\put (-152,151) {\scriptsize $z$}
\put (-70,140) {\scriptsize $\hat \eta_{z_0}$}
\put (-40,157) {\scriptsize $1$}
\put (-28,124) {\scriptsize $z_0$}
\put (-247,49) {\scriptsize $1$}
\put (-246,11) {\scriptsize $z_0$}
\put (-142,66) {\scriptsize $1$}
\put (-142,6) {\scriptsize $z_0$}
\put (-40,62) {\scriptsize $1$}
\put (-40,10) {\scriptsize $z_0$}
\put (-46,41) {\scriptsize $a$}
\caption{\footnotesize{Homotopical deformation of the integration path when $z_0$ turns around $1$.}} 
\end{figure}


We start by integrating Hadamard convolution formula  on a circle $\eta_{z_0}$ of radius $r$ with $|z_0|<r<1$, 
so that both series giving $F(u)$ and $G(z_0/u)$ 
are converging. When $z$, starting at $z_0$, moves around 
the point $1$, we deform homotopically the integration path 
into $\eta_z$ so that $z$ never crosses the integration path. With such condition, the convolution formula yields the analytic 
continuation of $F\odot G$. When we return to $z_0$, the path is deformed into $\hat \eta_{z_0}$ (see Figure 2)
the Hadamard product takes the value $(F\odot G)_+ (z_0)$. The 
difference 
$$
\Delta_1 (F\odot G)(z_0) = (F\odot G)_+ (z_0) -(F\odot G) (z_0)
$$
is the monodromy at $1$. According to the convolution formula, this difference can be computed by integrating on the homotopical difference of the 
two paths $\hat \eta_{z_0}-\eta_{z_0}$ (Figure 2). Considering the intersection point $a=[z_0,1]\cap \eta_{z_0}$, and 
shrinking the path as shown in Figure 2, 
this difference is composed by four (indeed only two, repeated twice) vertical segments (it is indeed a ``train track'')
$\eta_1=[a, 1]$, $\eta_2=[a,z_0]$, $\eta_3=[a, 1]$ and $\eta_4=[a,z_0]$ where we integrate in both directions different 
functions. Note also that the small turning loops around $1$ and $z_0$ give no contribution when they shrink because 
of the integrability condition (in the proof of the general case there is a non-trivial residual contribution here). 
We decompose further each path $\eta_j$ into two consecutive paths $\eta_j=\eta_j^-\cup \eta_j^+$ so that the difference
$\hat \eta_{z_0}-\eta_{z_0}$ decomposes as
$$
\hat \eta_{z_0}-\eta_{z_0}= \eta_1^-\cup \eta_1^+ \cup \eta_2^-\cup \eta_2^+ \cup \eta_3^-\cup \eta_3^+ \cup \eta_4^-\cup \eta_4^+
$$
where the paths $\eta_j^\pm$ follow each other in the order this union is written. 
We first list the functions that are integrated against 
the differential $du/u$ in each path. To compute these functions, notice that the monodromy around $u=1$ of $F(u)$ is $\Delta_1 F(u)$, 
and, according to Corollary \ref{cor:change_1/x},  the monodromy around $z_0$ of $G(z_0/u)$ is 
$\Delta_1 G(z_0/u)$, i.e.
$$
\Delta_{z_0} (G(z_0/u))=(\Delta_1 G)(z_0/u) \ .
$$
We need to take into account the sign corresponding to the orientation 
of the loop around each singularity.  We denote by $F_+ =F+\Delta_1 F$ and $G_+=G+\Delta_1 G$.
and we are assuming that 
the monodromies $\Delta_1 F$ and $\Delta_1 G$ are holomorphic at $1$.
\begin{align*}
\eta_1^- \rightarrow & F(u)G(z_0/u)  \\ 
\eta_1^+ \rightarrow & F_+(u)G(z_0/u) \\ 
\eta_2^- \rightarrow & F_+(u)G(z_0/u) \\ 
\eta_2^+ \rightarrow & F_+(u)G_+(z_0/u) \\ 
\eta_3^- \rightarrow & F_+(u)G_+(z_0/u) \\ 
\eta_3^+ \rightarrow & F(u)G_+(z_0/u) \\
\eta_4^- \rightarrow & F(u)G_+(z_0/u) \\ 
\eta_4^+ \rightarrow & F(u)G(z_0/u)  \\ 
\end{align*}
Now, for the contributions of each integral we have
$$
\int_{\eta_j}= \int_{\eta_j^-} -\int_{\eta_j^+}
$$
and these contributions for each integral $\eta_j$ are, respectively,
\begin{align*}
&\int_{\eta_1} - \Delta_1F(u) G(z_0/u) \, \frac{du}{u} = - \int_a^{1} \Delta_1 F(u) G(z_0/u) \, \frac{du}{u}\\
&\int_{\eta_2} -  F_+(u) \Delta_1 G(z_0/u)  \, \frac{du}{u} 
= - \int_{a}^{z_0} F_+(u) \Delta_1 G(z_0/u) \, \frac{du}{u}\\
&\int_{\eta_3} \Delta_1 F(u)  G_+(z_0/u) \, \frac{du}{u} 
=  \int_a^{1} \Delta_1 F(u)  G_+(z_0/u) \, \frac{du}{u} \\
&\int_{\eta_4}   F(u) \Delta_1 G(z_0/u)  \, \frac{du}{u} =  \int_{a}^{z_0}  F(u) \Delta_1 G(z_0/u)  \, \frac{du}{u} \\
\end{align*}
Adding up, we get
\begin{align*}
\int_{\eta_1} + \int_{\eta_3} &= \int_a^{1}  \Delta_1 F(u) \Delta_1 G(z_0/u) \, \frac{du}{u} \\
\int_{\eta_2} + \int_{\eta_4} &= -\int_{a}^{z_0}  \Delta_1 F(u) \Delta_1 G(z_0/u) \, \frac{du}{u} \\
\end{align*}
and finally
$$
\int_{\eta_{z_0}-\hat \eta_{z_0}} = \int_{\eta_1} + \int_{\eta_2} +\int_{\eta_3} + \int_{\eta_4} =-\int_1^{z_0}  \Delta_1 F(u) \Delta_1 G(z_0/u) \, \frac{du}{u}
$$
which gives the formula for the case of a single singularity.

In the general case when $\a\b=\g$, point $1$ is replaced by $\a$ and point $z_0$ is replaced by $z_0/\b$, and  we get the contribution 
$$
-\int_\a^{z_0/\b}  \Delta_1 F(u) \Delta_1 G(z_0/u) \, \frac{du}{u}
$$

In this path deformation argument, 
there is no residual contribution at $1$ because the functions and their monodromies are all integrable at $1$.

\subsection{Proof for higher multiplicity singularities.}

The initial Jordan loop of integration $\eta_{z_0}$ separates the singularities $(\alpha)$ of $F$, 
which are in the outside unbounded region, 
from the  singularities $(z_0/\beta)$ of $G(z_0/u)$, which are in the inside bounded region. 
As before, when the point $z$ starting at $z_0$ circles once around $\a\b =\g$, then all 
points $z/\beta$ circles once around the corresponding $\a$ in a synchronized choreography. 
We end-up with a path $\hat \eta_{z_0}$ as shown in Figure 1 (reversing the orientation). 
Then the difference of paths 
$\eta_{z_0} -\hat \eta_{z_0}$ is decomposed into a finite number of quadruple 
loops as the one considered before, one for each pair $(\a, \b)$ such that $\g=\a\b$. 
The total contribution adds the contribution of each quadruple loop and the formula follows.

\section{General holomorphic monodromy formula.}

\subsection{Formula for the Hadamard product.}
We prove the general result by following the ideas from the previous section. We are 
reduced to consider the case of a singularity $\a \b$ without multiplicity, and 
again we can assume $\a=\b=1$. We consider 
the same integration path. 

The only difference appears when we shrink loops at $u=1$ and $u=z_0$. For example, for the first path $\eta_1$, 
when we turn around $u=1$, 
we get an extra contribution of 
$$
\lim_{\epsilon_1 \to 1} \frac{1}{2\pi i}\int_{\epsilon_1}  F(u) G(z_0/u) \, \frac{du}{u}
$$
where $\epsilon_1$ is a local positive circle loop around $u=1$, and we take 
the limit when this loop converges to $u=1$.
We have
$$
\lim_{\epsilon_1 \to 1} \frac{1}{2\pi i}\int_{\epsilon_1}  F(u) G(z_0/u) \, \frac{du}{u} =\lim_{\epsilon_1 \to 1} \frac{1}{2\pi i}\int_{\epsilon_1}  F_0(u) G(z_0/u) \, \frac{du}{u} + \lim_{\epsilon_1 \to 1} \frac{1}{2\pi i}  \int_{\epsilon_1}  \frac{1}{2\pi i} \log(u-1) \Delta_1 F(u) G(z_0/u) \, \frac{du}{u}
$$
and the last  limit is zero because of the integrability condition, and more precisely because 
$\Delta_1 F(u) G(z_0/u)/u$ is holomorphic and $\log(u-1)$ integrable at $u=1$.

Therefore, for each path $\eta_j$ we get the following extra residue contributions to the integral 
(the orientations give the proper signs):

\begin{align*}
\eta_1 \rightarrow & \Res_{u=1} \left ( \frac{F_0(u) G(z_0/u)}{u} \right ) \\
\eta_2 \rightarrow & \Res_{u=z_0} \left ( \frac{F_+(u) G_0(z_0/u)}{u} \right ) \\ 
\eta_3 \rightarrow & -\Res_{u=1} \left ( \frac{F_0(u) G_+(z_0/u)}{u} \right ) \\ 
\eta_4 \rightarrow & - \Res_{u=z_0} \left ( \frac{F(u) G_0(z_0/u)}{u} \right )  
\end{align*}

Adding these four contributions gives a total residual contribution of

\begin{equation*}
R =- \Res_{u=1} \left ( \frac{F_0(u) \Delta_1 G(z_0/u)}{u} \right ) 
+\Res_{u=z_0} \left ( \frac{\Delta_1 F(u) G_0(z_0/u)}{u} \right ) \ \ .
\end{equation*}

We can give to this expression a symmetric form using the following elementary Lemma.

\begin{lemma}
We have 
$$
\Res_{u=z_0} \left ( \frac{\Delta_1 F(u) G_0(z_0/u)}{u} \right ) =- 
\Res_{u=1} \left ( \frac{G_0(u) \Delta_1 F(z_0/u)}{u} \right ) 
$$
\end{lemma}

\begin{proof}
We consider a local loop $\g$ with winding number $1$ with respect to $z_0$ and use the Residue Theorem
and the change of variables $v=z_0/u$, $du=-z_0 dv/v^2$, $\gamma'$ the image of $\gamma$ that is a local loop
with winding number number with respect to $1$,
\begin{align*}
\Res_{u=z_0} \left ( \frac{\Delta_1 F(u) G_0(z_0/u)}{u} \right ) 
&= \frac{1}{2\pi i} \int_\g  \frac{\Delta_1 F(u) G_0(z_0/u)}{u} \, du \\
&= \frac{1}{2\pi i} \int_{\g'}  \frac{\Delta_1 F(z_0/v) G_0(v)}{z_0/v} (-z_0) \, \frac{dv}{v^2} \\
&= -\frac{1}{2\pi i} \int_{\g'}  \frac{\Delta_1 F(z_0/v) G_0(v)}{v}  \, dv \\
&= -\Res_{u=1} \left ( \frac{G_0(u) \Delta_1 F(z_0/u)}{u} \right )
\end{align*}
\end{proof}

Finally, the residue total contribution is 
\begin{equation*}
R =- \Res_{u=1} \left ( \frac{F_0(u) \Delta_1 G(z_0/u)}{u} \right ) 
-\Res_{u=1} \left ( \frac{G_0(u) \Delta_1 F(z_0/u)}{u} \right ) 
\end{equation*}
In the case of a general singularity $\g=\a\b$, this residue contribution is
\begin{equation*}
R =- \Res_{u=\a} \left ( \frac{F_0(u) \Delta_\b G(z_0/u)}{u} \right ) 
-\Res_{u=\b} \left ( \frac{G_0(u) \Delta_\a F(z_0/u)}{u} \right ) 
\end{equation*}
and this gives the holomorphic monodromy formula (\ref{eq:general_monodromy_convolution}).

\subsection{Formula for the exponential e\~ne product.}

The proof for the exponential e\~ne product follows the same lines. 
The only difference is the the sign in the convolution formula and the integrating function. Since the monodromy operator 
commutes with differentiation there is almost no difference in the entire argument. 
We need to use the following Lemma:

\begin{lemma}\label{lem_F0_derivative}
A holomorphic singulartiy $\a$ of $F$ is a holomorphic singularity for $F'$ and we have
$$
\left (F'\right )_0=F_0'+\frac{1}{2\pi i} \, \frac{1}{z-\a} \, \Delta_\a F
$$
\end{lemma}

\begin{proof}
We use the commutation of the monodromy operator with the differential operator. 
We differentiate the decomposition
$$
F(z)=F_0(z)+\frac{1}{2\pi i} \, \log (z-\a) \Delta_\a F (z)
$$
and obtain
$$
F'(z)=\left ( F_0'(z)+\frac{1}{2\pi i} \, \frac{1}{ z-\a} \Delta_\a F (z) \right ) +
\frac{1}{2\pi i} \, \log (z-\a) \Delta_\a F' (z)
$$
and by uniqueness of the monodromy decomposition for $F'$ from Proposition \ref{prop_decomposition}, we get the result.
\end{proof}

The same proof as before, starting from the exponential e\~ne product convolution formula, gives

\begin{align*}
\Delta_{\g} (F\star_e G) (z)&= \sum_{\substack{\a ,\b \\ \a\b=\g}} 
 \Res_{u=\a} \left ( (F')_0(u) \Delta_\b G(z/u) \right ) 
+\sum_{\substack{\a ,\b \\ \a\b=\g}} \Res_{u=\b} \left ( G_0(u) \Delta_\a F'(z/u) \right )  \\
& \ \ \ \, + \frac{1}{2\pi i}  \sum_{\substack{\a ,\b \\ \a\b=\g}}  
\int_\a^{z/\b} \Delta_{\a} F'(u) \Delta_\b G (z/u) \, du \nonumber
\end{align*}

Now, using the Lemma \ref{lem_F0_derivative}, we have

\begin{align*}
 \Res_{u=\a} \left ( (F')_0(u) \Delta_\b G(z/u) \right ) &= \Res_{u=\a} \left ( F'_0(u) \Delta_\b G(z/u) \right )
+\frac{1}{2\pi i}\, \Res_{u=\a} \left ( \frac{1}{u-\a} \Delta_\a F(u) \Delta_\b G(z/u) \right )\\
&= \Res_{u=\a} \left ( F'_0(u) \Delta_\b G(z/u) \right ) +\frac{1}{2\pi i}  \Delta_\a F(\a) \Delta_\b G(z/ \a)  
 \end{align*}
and formula (\ref{eq:ene_monodromy_convolution}) follows.

This commutation property for the derivation is also the reason for the symmetry on $F$ and $G$ 
in the formula. By integration by parts we have:

\begin{proposition}\label{prop:symmetry_ene}
For $\g=\a\b$ we have
 \begin{equation*} 
 \Delta_\a F(\a) \Delta_\b G(z/ \a) +\int_\a^{z/\b} \Delta_{\a} F'(u) \Delta_{\b} G (z/u) \, du = 
 \Delta_\a F(z/\b) \Delta_\b G(\b) +\int_\b^{z/\a} \Delta_{\b} G'(u) \Delta_{\a} F (z/u) \, du
\end{equation*}
\end{proposition}

\begin{proof}
We perform an integration by parts, and the change of variables $v=z/u$,

\begin{align*}
\int_\a^{z/\b} \Delta_{\a} F'(u) \Delta_{\b} G (z/u) \, du &= \left [ \Delta_{\a} F(u) \Delta_{\b} G (z/u)\right ]_{\a}^{z/\b} - \int_\a^{z/\b} \Delta_{\a} F(u) \Delta_{\b} G' (z/u) \left (-\frac{1}{u^2}\right )\, du \\
&= \left [ \Delta_{\a} F(u) \Delta_{\b} G (z/u)\right ]_{\a}^{z/\b}  + \int_\b^{z/\a} \Delta_{\a} F(z/v) \Delta_{\b} G' (v) \, dv \\
\end{align*}

\end{proof}

We could have derived the formula for the monodromy of the exponential e\~ne product 
from formula (\ref{eq:general_monodromy_convolution}) for the Hadamard product by using 
$$
F\star_e G =-K_0\odot F\odot G \ .
$$
It is instructive to derive it. We have already noted that the effect on the monodromy of the 
Hadamard product with $-K_0$ is just the derivation of the monodromy, hence
$$
\Delta_{\g} (F\star_e G)= \left ( \Delta_\g F\odot G \right )'  
$$
thus, we only need to take the derivate of formula (\ref{eq:general_monodromy_convolution}). 
Using the commutation 
with the monodromy operator and previous Lemma \ref{lem_F0_derivative} we get the result again.
The proof of Corollary \ref{cor:ene_improved_borel} is straightforward 
from the e\~ne monodromy formula.

\section{Applications.}

\subsection{Application 1: Monodromy of polylogarithms.} \label{sec:polylogarithms}

We use formula (\ref{eq:general_monodromy_convolution}) to compute the classical monodromy of polylogarithms.  This is usually done in the literature by using 
functional equations (as for example in \cite{Oesterle1993}).
The polylogarithm $\Li_k(z)$ for $k=1,2,\ldots$ can be defined for $|z|<1$ by the converging series
$$
\Li_k(z) =\sum_{n=1}^{+ \infty}  n^{-k} \, z^n
$$
Therefore $\Li_1$ is given by the logarithm, $\Li_1(z)=-\log(1-z)$ that has the principal 
branch with a multivalued holomorphic extension to $\CC-\{1\}$ with a unique singularity 
at $1$ which is of logarithmic type with a constant monodromy
$$
\Delta_1 \Li_1 = -2\pi i \ .
$$
Higher polylogarithms can also be defined inductively by integration
$$
\Li_{k+1}(z)=\int_0^z \Li_k(u)\, \frac{du}{u}
$$
as we readily see by integrating term by term the defining power series in its disk of convergence.
Using the integral expression we prove by induction that their principal branch extends holomorphically to a multivalued function on $\CC-\{1\}$ with a unique 
singularity at $1$ (non-principal branches can have also singularities at $0$ as we see below).

From the power series definition it is clear that higher order polylogarithms can also 
be defined inductivelly using Hadamard multiplication
$$
\Li_{k+1} =\Li_k\odot \Li_1 \ .
$$
The minimal Hadamard ring structure (for the sum and the Hadamard product) generated  by the logarithm is obtained by adjoining higher order polylogarithms, for $k,l\geq 1$,
$$
\Li_{k+l} =\Li_k\odot \Li_l \ .
$$
As a Corollary of monodromy formula (\ref{eq:general_monodromy_convolution}) we can compute directly the monodromy at $s=1$.

\begin{corollary}
For $k\geq 2$ the only singularities of the analytic continuation of $\Li_k$ are located 
at $0$ and $1$.
For $k\geq 1$ the monodromy at $1$ of the principal branch of $\Li_k$ is holomorphic at $z=1$
and, more precisely, 
$$
\Delta_1 \Li_k =-\frac{2\pi i}{(k-1)!} \, (\log z)^{k-1}
$$
\end{corollary}

\begin{proof}
For $k=1$ this is the monodromy of the classical logarithm that is holomorphic (constant) at $1$. 
Assuming by induction the result for $k\geq 1$, the monodromy for $\Li_k$ 
is holomorphic at $z=1$, and we can use Theorem \ref{thm:holomorphic} with the 
formula $\Li_{k+1} =\Li_k\odot \Li_1$, observing that the singularity at $\a=1$ is totally holomorphic, 
so we can use the simpler formula (\ref{eq:monodromy_convolution}), and we get (using the change of variables $v=\log u$)
\begin{align*}
\Delta_1 \Li_{k+1} (z)&=-\frac{1}{2\pi i} \int_1^z -\frac{2\pi i}{(k-1)!} 
(\log u)^{k-1} (-2 \pi i) \, \frac{du}{u}\\
&= -\frac{2\pi i}{(k-1)!}  \int_0^{\log z} v^{k-1}\, dv\\
&= -\frac{2\pi i}{k!} (\log z)^k
\end{align*}
\end{proof}
We can use the formula for the exponential e\~ne product to cross check 
this result  since it follows from the power series expansion that, for $k,l\geq 1$,
\begin{equation*}
\Li_k\star_e \Li_l =-\Li_{k+l-1}  
\end{equation*}
and in particular,
\begin{align*}
\Li_{k+1} &= -\Li_k\star_e \Li_2 \\
\Li_{1}&= -\Li_1\star_e \Li_1   
\end{align*}
hence the minimal exponential e\~ne ring containing $\Li_1$ is just generated by $\Li_1$, but the one 
containing $\Li_1$ and $\Li_2$ is generated by all the higher polylogarithms as before.
Related to this we observe that we need to know the monodromy of $\Li_2$ to start the induction with the 
exponential e\~ne product formula.
The exponential e\~ne convolution formula and the Hadamard convolution formula are
directly related using $\Li'_{k+1}(z)=\Li_k(z)/z$ since 
$$
\Li_k\star_e \Li_l=-\frac{1}{2\pi i}\int_\eta \Li_k'(u)\Li_l(z/u)\, du 
= -\frac{1}{2\pi i}\int_\eta \Li_k(u)\Li_l(z/u)\, \frac{du}{u} = -\Li_k\odot \Li_l =-\Li_{k+l}
$$

Now, using the formula (\ref{eq:ene_tot_hol_monodromy_convolution}) for the e\~ne monodromy for the totally 
holomorphic singularity at $\a=1$, 
and the change of variables $v=\log u/\log z$, 
we have
\begin{align*}
\Delta_1 (\Li_k\star_e \Li_l) &= \frac{1}{2\pi i} \Delta_1 \Li_k(1) \Li_l (z)+\frac{1}{2\pi i} \int_1^z \left (\frac{-2\pi i}{(k-1)!} 
(\log u)^{k-1}\right )'\left ( \frac{-2 \pi i}{(l-1)!} (\log (z/u))^{l-1}\right )  du \\
&= 0+\frac{2\pi i}{(k-2)!(l-1)!}  \int_1^z  
(\log u)^{k-2} .  (\log (z/u))^{l-1}  du \\
%
&=\frac{2\pi i}{(k-2)!(l-1)!} (\log z)^{k+l-2} \int_0^1 v^{k-2} (1-v)^{l-1} \, dv \\ 
&= \frac{2\pi i}{(k-2)!(l-1)!} (\log z)^{k+l-2} B(k-1, l)\\
&= \frac{2\pi i}{(k-2)!(l-1)!} (\log z)^{k+l-2} \frac{\Gamma(k-1) \Gamma(l)}{\Gamma(k+l-1)}\\
&=\frac{2\pi i}{(k+l-2)!} (\log z)^{k+l-2} \\
&=-\Delta_1 (\Li_{k+l-1})
\end{align*}

The polylogarithm ring is a good example showing  that the e\~ne ring structure corresponds 
to the twisted Hadamard structure (see \cite{PM1} Section 10). We note also that the monodromy formula 
suggests that the proper normalization of the polylogarithm functions is
$$
\li_k(z) =-\frac{1}{2\pi i} \Li_k(z)
$$
so that
$$
\Delta_1 \li_k = \frac{1}{(k-1)!} (\log z)^{k-1} \ .
$$

\subsection{Application 2: Polynomial logarithmic monodromy  class.} \label{sec:polylogarithm_class}

The simplest case of holomorphic monodromy occurs for a polynomial monodromy. 
We determine now the minimal Hadamard ring 
containing polynomial monodromies. As a Corollary of our main formula we have the following Proposition:

\begin{proposition}\label{prop:pol_monodromy}
Let $F$ and $G$ be two holomorphic functions with polynomial monodromies, $\Delta_\a F, \Delta_\beta G \in \CC [z]$.
Then we have that 
$$
\Delta_\g (F\odot G) \in \CC[z] \oplus \CC[z] \log z \ .
$$
\end{proposition}

\begin{proof}
We write
\begin{equation*}
\Delta_\a F (z) = \sum_n a_n z^n \ \ \ \ \ \text{and} \ \ \ \ \ \Delta_\b G (z) = \sum_m b_m z^m 
\end{equation*}
then, using the holomorphic monodromy formula, we get
\begin{align*}
\Delta_\g (F\odot G) (z) &=-\frac{1}{2\pi i}\sum_{\substack{\a ,\b \\ \a\b=\g}} \int_\a^{z/\b} 
\Delta_\a F (u) \Delta_\b G (z/u) \, \frac{du}{u} \\
&= -\frac{1}{2\pi i}\sum_{\substack{\a ,\b \\ \a\b=\g}} \sum_{n,m} a_n b_m z^m \int_\a^{z/\b} u^{n-m-1} \, du \\
&= -\frac{1}{2\pi i}\sum_{\substack{\a ,\b \\ \a\b=\g}} \sum_{n,m, n\not= m} a_n b_m z^m \, 
\frac{z^{n-m-1} - \gamma^{n-m-1}}{\b^{n-m-1}(n-m)}  \\
&\ \ \ -\frac{1}{2\pi i} \sum_{\substack{\a ,\b \\ \a\b=\g}}\sum_{n} a_n b_n z^n (\log z-\log \g) 
\end{align*}
\end{proof}
We have the following elementary Lemma:

\begin{lemma}
Let $\alpha\in \CC^*$, and integers $k,l\geq 0$. We have
$$
\int_\a^z u^k (\log u)^l \,  du \in \QQ[\a, \log \a][z,\log z] \ .
$$
\end{lemma}

\begin{proof}
The result is obtained by recurrence on the exponent $l\geq 0$ and integration by parts.
\end{proof}

As a Corollary we obtain the stability of the ring $\CC[z,\log z]$ by the holomorphic monodromy formula, and the same computation 
as for Proposition \ref{prop:pol_monodromy} proves the general Theorem:

\begin{theorem}
Let $F$ and $G$ be two holomorphic functions with monodromies, $\Delta_\a F, \Delta_\beta G \in \CC [z, \log z]$.
Then we have that 
$$
\Delta_\g (F\odot G) \in \CC[z, \log z] \ .
$$
\end{theorem}

This motivates the definition of the \textit{polynomial logarithmic monodromy class}.

\begin{definition}
The polynomial logarithmic monodromy (or $PLM$)class  is the class of holomorphic functions in a neighborhood of $0$ having only singularities with monodromies in the 
ring $\CC[z,\log z]$.
\end{definition}

We have proved:

\begin{proposition}
The $PLM$ class is closed under the Hadamard product, and it is the minimal Hadamard ring containing the subclass of functions with polynomial 
monodromies.
\end{proposition}

Note that we have the same result for the e\~ne product. We can be more precise keeping track of the previous computations.

\begin{definition}
Let $K\subset \CC$ be a field such that $2\pi i\in K$.
The $PLM(K)$ class is composed by holomorphic functions in a neighborhood of $0$, 
having only singularities in the ring $K[z, \log z]$, with singularities $\a \in K$.
\end{definition}

\begin{theorem}
The $PLM(K)$ class is closed under the Hadamard product, and it is the minimal 
Hadamard ring containing the subclass of functions with polynomial monodromies in $K[z]$ with singularities in $K$.
\end{theorem}

\subsection{Application 3: Divisor interpretation of the e\~ne product.} \label{sec:divisor}

As explained in the introduction, the exponential e\~ne product linearizes is the exponential form 
of the e\~ne product. It is remarkable that the formulas for the e\~ne product are linearized 
through the exponential function.

Using the monodromy formulas we can prove directly the divisor interpretation of the e\~ne product. 
Thus this gives an alternative definition of the e\~ne product and its properties. 
Starting from two meromorphic functions in the plane, holomorphic at $0$, normalized such that 
$f(0)=g(0)=1$, we can consider their exponential form:
\begin{align*}
f(z) &=\exp (F(z))= 1+a_1 z+a_2 z^2+\ldots =1+\sum_{n\geq  1} a_n \, z^n\\ 
g(z) &=\exp (G(z))=1+b_1 z+b_2 z^2+\ldots =1+\sum_{n\geq 1} b_n \, z^n
\end{align*}
The functions $F=\log f$ and $G=\log g$ are  holomorphic germs at $z=0$. Their singularities are located 
at the zeros of $f$ and $g$ respectively. We have constant monodromies,
\begin{align*}
\Delta_\a F &= 2\pi i \, n_\a \\
\Delta_\b G &= 2\pi i \, n_\b 
\end{align*}
where $n_\a$, resp. $n_\b$, is the multiplicity of the zero or pole of $f$, resp. $g$ (negative multiplicity for poles).  
We can define the e\~ne product by the exponential e\~ne product, 
$$
f\star g =\exp (F\star_e G)
$$
this defines a holomorphic germ near $0$, with $f\star g(0)=1$. We know that the singularities of $f\star g$ 
are located at the singularities of $F\star_e G$, i.e. at the points $\g=\a\b$.

The monodromies are totally holomorphic, hence we can use formula (\ref{eq:ene_tot_hol_monodromy_convolution}) and compute the monodromies of the singularities $\g$,
which gives
$$
\Delta_\g (F\star_e G) =\frac{1}{2\pi i} \sum_{\substack{ \a\b=\g}}  
(2\pi i \, n_\a) (2\pi i \, n_\b) = (2\pi i) \sum_{\substack{\a\b=\g}} n_\a n_\b 
$$
This proves that $\g$ is a zero or pole for $f\star g$ of multiplicity 
$ 
n_\g =\sum n_\a n_\b 
$
hence recovering the definition given in \cite{PM1}. It is instructive to note, in view of the infinite divisor interpretation of 
the e\~ne product given in \cite{PM2}, that isolated essential singularities of $f$ and $g$ 
correspond to regular poles of $F$ and $G$, which are a particular case of regular isolated 
singularities. 

Using the analogue of Borel Theorem for regular isolated singularities, 
we see that the divisor e\~ne product interpretation extends further to singularities 
of $f$ and $g$ of infinite order, for example of the form
$f(z)=\exp(1-\exp (\frac{z}{z-1}))$.
The study of these ``higher order divisor'' extensions is left for future work.

\end{document}